\documentclass[reqno,11pt]{amsart}

\usepackage{amsthm,amsmath,amssymb,bm}
\usepackage{mathtools}
\usepackage{pifont}
\usepackage{tikz}
\usepackage{xcolor}

\usepackage{geometry}
\usepackage[colorlinks=true,
linkcolor=blue,citecolor=blue,
urlcolor=blue]{hyperref}

\usepackage{microtype}

\numberwithin{equation}{section}

\newtheorem{theorem}{Theorem}[section]
\newtheorem{proposition}[theorem]{Proposition}
\newtheorem{lemma}[theorem]{Lemma}
\newtheorem{claim}[theorem]{Claim}

\newtheorem{conjecture}[theorem]{Conjecture}

\theoremstyle{definition}

\theoremstyle{remark}

\newcommand{\dif}{\mathop{}\!\mathrm{d}}

\allowdisplaybreaks
\title[Two conjectures in spectral graph theory]{Two conjectures in spectral graph theory involving the linear combinations of graph eigenvalues}

\author[L. Liu]{Lele Liu}

\thanks{
This author is supported by the National Natural Science Foundation of China (No. 12001370).
}

\address{College of Science, University of Shanghai for Science and Technology, Shanghai 200093, China}
\email{ahhylau@outlook.com}

\begin{document}

\begin{abstract}
We prove two conjectures in spectral extremal graph theory involving the linear combinations of graph eigenvalues.
Let $\lambda_1(G)$ be the largest eigenvalue of the adjacency matrix of a graph $G$, and $\overline{G}$ be the 
complement of $G$. A nice conjecture states that the graph on $n$ vertices maximizing $\lambda_1(G) + \lambda_1(\overline{G})$ 
is the join of a clique and an independent set, with $\lfloor n/3\rfloor$ and $\lceil 2n/3\rceil$ (also $\lceil n/3\rceil$ 
and $\lfloor 2n/3\rfloor$ if $n \equiv 2 \pmod{3}$) vertices, respectively. We resolve this conjecture for 
sufficiently large $n$ using analytic methods. Our second result concerns the $Q$-spread $s_Q(G)$ of a graph 
$G$, which is defined as the difference between the largest eigenvalue and least eigenvalue of the signless 
Laplacian of $G$. It was conjectured by Cvetkovi\'c, Rowlinson and Simi\'c in $2007$ that the unique 
$n$-vertex connected graph of maximum $Q$-spread is the graph formed by adding a pendant edge to $K_{n-1}$. 
We confirm this conjecture for sufficiently large $n$.
\end{abstract}


\maketitle

\section{Introduction}

Spectral extremal graph theory seeks to maximize or minimize some function of graph eigenvalues or eigenvectors over 
a given family of graphs. In this paper we focus on two problems in spectral extremal graph theory. Before the 
statement of the results, we introduce some notation and definitions. Consider a simple graph $G = (V,E)$ with 
vertex set $V = \{v_1,v_2,\ldots,v_n\}$. The \emph{adjacency matrix} of $G$ is defined to be a matrix $A(G)=[a_{ij}]$ 
of order $n$, where $a_{ij}=1$ if $v_i$ is adjacent to $v_j$, and $a_{ij}=0$ otherwise. The \emph{signless Laplacian} 
of $G$ is defined by $Q(G) := D(G) + A(G)$, where $D(G)$ is the diagonal matrix whose entries are the degrees of the 
vertices of $G$. The eigenvalues of $A(G)$ and $Q(G)$ are denoted by $\lambda_1(G)\geq \cdots\geq\lambda_n(G)$ and 
$q_1(G)\geq \cdots\geq q_n(G)$, respectively. Recall that a \emph{complete split graph} with parameters $n$, 
$\omega$ $(\omega\leq n)$, denoted by $CS_{n,\omega}$, is the graph on $n$ vertices obtain from a clique on 
$\omega$ vertices and an independent set on the remaining $n-\omega$ vertices in which each vertex of the clique is
adjacent to each vertex of the independent set. We study the following two conjectures.

\begin{conjecture}[\cite{AouchicheBell2008,AouchicheHansen2010,AouchicheCaporossiHansen2013,AouchicheHansen2013,ElphickAouchiche2017}]
\label{conj:NG-conjecture}
Let $G$ be a graph on $n$ vertices. Then
\[
\lambda_1(G) + \lambda_1(\overline{G}) \leq \frac{4}{3} n - \frac{5}{3} - \frac{1}{6}
\begin{dcases}
3n-1-\sqrt{9n^2-6n+9}, & n \equiv 0 \hspace{-2.5mm} \pmod{3}, \\
3n-2-\sqrt{9n^2-12n+12}, & n \equiv 1 \hspace{-2.5mm} \pmod{3}, \\
0, & n \equiv 2 \hspace{-2.5mm} \pmod{3}.
\end{dcases}
\]
Equality holds if and only if $G$ \textup{(}or $\overline{G}$\textup{)} is the complete split graph with a clique  
on $\lfloor n/3\rfloor$ vertices, and also on $\lceil n/3\rceil$ vertices if $n \equiv 2 \pmod{3}$.
\end{conjecture}

\begin{conjecture}[\cite{CvetkovicRowlinson2007,AouchicheHansen2010,Oliveira2010}]\label{conj:Q-spectral-spread-conjecture}
Over all connected graphs on $n\geq 6$ vertices, $s_Q(G)$ is maximized by $K_{n-1}^+$, the graph formed by adding a pendant edge to 
the complete graph $K_{n-1}$. Furthermore, $s_Q(G)$ is minimized by the path $P_n$ and, in the case that $n$ is odd, by the cycle $C_n$. 
\end{conjecture}

It is worth mentioning that the second part of Conjecture \ref{conj:Q-spectral-spread-conjecture} was proved by 
Das \cite{Das2012} and Fan et al. \cite{FanFallat2012} independently using different proof techniques.

In this paper we confirm Conjecture \ref{conj:NG-conjecture} in Section \ref{sec3} and the first part of 
Conjecture \ref{conj:Q-spectral-spread-conjecture} in Section \ref{sec4} for sufficiently large $n$.

\subsection{Nordhaus--Gaddum type inequality}

The study of Nordhaus--Gaddum type inequality has a long history, dating back to a classical paper of Nordhaus and Gaddum 
\cite{NordhausGaddum1956}. The following inequalities for the chromatic numbers $\chi(G)$ and $\chi(\overline{G})$ has been 
established.
\[
2\sqrt{n} \leq\chi(G) + \chi(\overline{G}) \leq n+1, 
\]
\[
n \leq \chi(G) \cdot \chi(\overline{G}) \leq \frac{(n+1)^2}{4},
\] 
where $\overline{G}$ is the complement of $G$. Since then, any bound on the sum or the product of an invariant in a graph $G$ 
and the same invariant in $\overline{G}$ is called a \emph{Nordhaus--Gaddum type inequality}. For an exhaustive survey of such 
relations see Aouchiche and Hansen \cite{AouchicheHansen2013} and the references therein. Many of those inequalities involve 
eigenvalues of adjacency, Laplacian and signless Laplacian matrices of graphs. The first known spectral Nordhaus--Gaddum results 
belong to Nosal \cite{Nosal1970}, and to Amin--Hakimi \cite{AminHakimi1972}, who proved that for each graph $G$ of order $n$,
\[
\lambda_1(G) + \lambda_1(\overline{G}) < \sqrt{2} (n-1).
\]
Nikiforov \cite{Nikiforov2007} enhanced the above bound to $(\sqrt{2} - c)n$, where $c$ is some constant not less than $10^{-7}$. 
In the same paper, Nikiforov also conjectured that
\[
\lambda_1(G) + \lambda_1(\overline{G}) \leq \frac{4}{3} n + O(1).
\]
In \cite{Csikvari2009}, Csikv\'ari improved Nikiforov's bound to $(1+\sqrt{3})n/2 - 1$. Later, Terpai \cite{Terpai2011} 
resolved Nikiforov's conjecture using graph limit method.

Let us note that there are also some Nordhaus--Gaddum type inequalities for (signless) Laplacian eigenvalues of graphs. 
Most of the results on Laplacian eigenvalues are around the following conjecture which was posed in \cite{YouLiu2012,ZhaiShuHong2011}: 
for any graph $G$ on $n$ vertices,
\[
\mu_1(G) + \mu_1(\overline{G}) \leq 2n-1,
\]
with equality if and only if $G$ or $\overline{G}$ is isomorphic to the join of an isolated vertex and a disconnected 
graph of order $n-1$. Here, $\mu_1(G)$ and $\mu_1(\overline{G})$ are the largest eigenvalues of Laplacian matrices $L(G)$ 
and $L(\overline{G})$, respectively. Recently, Einollahzadeh and Karkhaneei \cite{Einollahzadeh2021} resolve the conjecture.
For signless Laplacian eigenvalues, Ashraf and Tayfeh-Rezaie \cite{AshrafTayfeh-Rezaie2014} prove that
$q_1(G) + q_1(\overline{G}) \leq 3n-4$ and equality holds if and only if $G$ is the star $K_{1,n-1}$, which confirms a 
conjecture in \cite{AouchicheHansen2013}.

We now turn to our topic. Using AutoGraphiX, Aouchiche et al. \cite{AouchicheBell2008} posed Conjecture \ref{conj:NG-conjecture} 
which also appeared in \cite{AouchicheHansen2010,AouchicheCaporossiHansen2013,AouchicheHansen2013,ElphickAouchiche2017}.
As mentioned above, Terpai \cite{Terpai2011} proved a tight upper bound on $\lambda_1(G) + \lambda_1(\overline{G})$ using 
analytical methods.  To be precise, he showed that
\[
\lambda_1(G) + \lambda_1(\overline{G}) \leq \frac{4}{3} n - 1.
\]
According to our knowledge, this is the best result until now.

\subsection{$Q$-spread of graphs}

The \emph{spread} $s(M)$ of an arbitrary $n\times n$ matrix $M$ is defined as the maximum modulus over the difference 
of all pairs of eigenvalues of $M$. When considering the adjacency matrix of a graph $G$, the spread is simply the 
distance between $\lambda_1$ and $\lambda_n$, denoted by $s(G):= \lambda_1(G) - \lambda_n(G)$. The adjacency spread of a 
graph has received much attention. In \cite{GregoryHershkowitz2001}, the authors investigated a number of properties 
regarding the spread of a graph, determining upper and lower bounds on $s(G)$. Recently, Breen et al. \cite{BreenRiasanovksy2022} 
determine the unique graph maximizing the spread among all $n$-vertex graphs, which resolves a conjecture in \cite{GregoryHershkowitz2001}.

With regard to the spread for signless Laplacian, we denote $s_Q(G):= q_1(G) - q_n(G)$ and refer it as \emph{$Q$-spread} of $G$.
Oliveira et al. \cite{Oliveira2010} present some upper and lower bounds for signless Laplacian spread. Liu et al. \cite{LiuLiu2010} 
determine the unique graph with maximum signless Laplacian spread among the class of connected unicyclic graphs of order 
$n$. Das \cite{Das2012} and Fan et al. \cite{FanFallat2012} confirm the second part of Conjecture \ref{conj:Q-spectral-spread-conjecture}
independently. We prove the first part for $n$ large enough.

\section{Preliminaries}

\subsection{Notation}

Given a subset $X$ of the vertex set $V(G)$ of a graph $G$, the subgraph of $G$ induced by $X$ is denoted by $G[X]$, and 
the graph obtained from $G$ by deleting $X$ is denoted by $G\setminus X$. We use $E(X)$ to denote the set of edges in the 
induced subgraph $G[X]$. As usual, for a vertex $v$ of $G$ we write $d_G(v)$ and $N_G(v)$ for the degree of $v$ and the 
set of neighbors of $v$ in $G$, respectively. If the underlying graph $G$ is clear from the context, simply $d(v)$ and $N(v)$. 
We denote the maximum degree and minimum degree of $G$ by $\Delta(G)$ and $\delta(G)$, respectively. For a vector 
$\bm{x} \in \mathbb{R}^n$, we denote by $x_{\max}:=\{|x_i|: 1\leq i\leq n\}$.

% Next, we collect some results that will be used in Section \ref{sec4}.
% The following inequality can be traced back to Merris \cite{Merris1998}.

% \begin{lemma}[\cite{Merris1998}]\label{lem:q1-upper-bound}
% Let $G$ be a graph. Then 
% \[
% q_1(G) \leq\max_{u\in V(G)} \Big\{ d(u) + \frac{1}{d(u)} \sum_{v\in N(u)} d(v) : d(u) > 0 \Big\}.
% \]
% \end{lemma}

% \begin{lemma}[\cite{Das2004}]\label{lem:average-2-degree-upper-bound}
% Let $G$ be a graph with $n$ vertices and $m$ edges. Then 
% \[
% \max_{u\in V(G)} \Big\{ d(u) + \frac{1}{d(u)} \sum_{v\in N(u)} d(v) : d(u) > 0 \Big\} \leq \frac{2m}{n-1} + n-2.
% \]
% \end{lemma}

\subsection{Graphons and graphon operators}

A \emph{graphon} is a symmetric measurable function $W: [0,1]^2 \to [0,1]$ (by symmetric, we mean 
$W(x,y)=W(y,x)$ for all $(x,y)\in [0,1]^2$). Let $\mathcal{W}$ denote the space of all graphon. 
Given a graph $G$ with $n$ vertices labeled $1,2,\ldots,n$, we define its \emph{associated graphon}
$W_G : [0, 1]^2 \to [0, 1]$ by first partitioning $[0, 1]$ into $n$ equal-length intervals 
$I_1,I_2,\ldots, I_n$ and setting $W_G$ to be $1$ on all $I_i \times I_j$ where $ij$ is an edge 
of $G$, and $0$ on all other $I_i \times I_j$'s. Clearly, associated graphon is not unique due 
to different labeling of vertices. Using an equivalence relation on $\mathcal{W}$ derived from 
the so-called \emph{cut metric}, we can identify associated graphons that are equivalent up to 
relabelling, and up to any differences on a set of measure zero. The \emph{cut norm} of a measurable 
function $W: [0,1]^2\to\mathbb{R}$ is defined as 
\[
\|W\|_{\square} := \sup_{S,\,T\subset [0,1]} \Big|\int_{S\times T} W(x,y) \dif x \dif y\Big|,
\]
where the supremum is taken over all measurable subsets $S$ and $T$. Given two symmetric measurable 
functions $U, W: [0,1]^2 \to\mathbb{R}$, we define their \emph{cut metric} to be 
\[
\delta_{\square} (U,W) := \inf_{\phi} \|U - W^{\phi}\|_{\square},
\]
where the infimum is taken over all invertible measure preserving maps $\phi: [0,1]\to [0,1]$, and 
$W^{\phi} (x,y):= W(\phi(x), \phi(y))$. For more results on graphons, we refer to \cite{Lovasz2012} 
for details.
% Let $\widetilde{\mathcal{W}}$ be the set of graphons where any pair of graphons with cut distance zero are 
% considered the same point in the space. This is a metric space under cut metric $\delta_{\square}$. We view 
% every graph $G$ as a point in $\widetilde{\mathcal{W}}$ via its associated graphon (note that several graphons 
% can be identified as the same point in $\widetilde{\mathcal{W}}$).

With $W\in\mathcal{W}$ we associate a graphon operator acting on $\mathcal{L}^2 [0,1]$, defined 
as the linear integral operator
\[
(A_W f) (x) := \int_0^1 W(x,y) f(y) \dif y
\]
for all $f\in\mathcal{L}^2 [0,1]$. Since $W$ is symmetric and bounded, $A_W$ is a compact Hermitian 
operator. In particular, $A_W$ has a discrete, real spectrum whose only possible accumulation point 
is zero.

Let $\mu(W)$ be the maximum eigenvalue of $A_W$. The following proposition establish a connection 
between $\lambda_1(G)$ and $\mu(W_G)$.

\begin{proposition}[\cite{Terpai2011}]\label{prop:relation}
Let $G$ be a graph on $n$ vertices. Then $\lambda_1(G) = n\cdot\mu(W_G)$.
\end{proposition}

\begin{proposition}[{\cite[Theorem 11.54]{Lovasz2012}}]\label{prop:converge}
Let $\{W_n\}$ be a sequence of graphons converging to $W$ with respect to $\delta_{\square}$. Then $\lim_{n\to\infty} \mu(W_n)=\mu(W)$.
\end{proposition}

\section{Nordhaus--Gaddum type inequality for spectral radius of graphs}
\label{sec3}

The aim of this section is to give a proof of Conjecture \ref{conj:NG-conjecture}. Throughout this 
section, we always assume that $G$ is a graph such that $p(G)$ attaining the maximum among all graphs 
on $n$ vertices, where $p(G) := \lambda_1(G)+\lambda_1(\overline{G})$. Observe that at least one of 
$G$ and $\overline{G}$ is connected. Without loss of generality, we assume $G$ is connected.
Let $\bm{x}$ and $\overline{\bm{x}}$ be the nonnegative eigenvectors with unit length corresponding 
to $\lambda_1(G)$ and $\lambda_1(\overline{G})$, respectively. For convenience, we set $\lambda_1:=\lambda_1(G)$ 
and $\overline{\lambda}_1 := \lambda_1(\overline{G})$.

First we prove a number of properties of $G$.

\begin{lemma}\label{lem:adjacent-iff}
For any vertices $u$ and $v$, $x_ux_v - \overline{x}_u \overline{x}_v > 0$ if and only if 
$u$ and $v$ are adjacent.
\end{lemma}

\begin{proof}
We first show that $x_ux_v - \overline{x}_u\overline{x}_v \neq 0$ by contradiction.
Denote by $H$ the graph formed by adding or deleting the edge $uv$ from $G$. By Rayleigh principle we see 
\begin{align*}
\lambda_1(H) 
& \geq \bm{x}^{\mathrm{T}} A(H) \bm{x} = \lambda_1(G) + 2(-1)^{a_{uv}} \cdot x_ux_v, \\
\lambda_1(\overline{H}) 
& \geq  \overline{\bm{x}}^{\mathrm{T}} A(\overline{H}) \overline{\bm{x}} 
= \lambda_1(\overline{G}) - 2(-1)^{a_{uv}} \cdot \overline{x}_u \overline{x}_v.
\end{align*}
Therefore, if $x_ux_v - \overline{x}_u\overline{x}_v = 0$, then
\begin{equation}\label{eq:p(H)-vs-p(G)}
p(H) = \lambda_1(H) + \lambda_1(\overline{H}) \geq \lambda_1(G) + \lambda_1(\overline{G}) 
+ 2 (-1)^{a_{uv}} \cdot (x_ux_v - \overline{x}_u\overline{x}_v) = p(G).
\end{equation}
Hence, each inequality above must be equality. So $\bm{x}$, $\overline{\bm{x}}$ are also eigenvectors of 
$H$ and $\overline{H}$ corresponding to $\lambda_1(H)$ and $\lambda_1(\overline{H})$, respectively. 
In particular, $A(H) \bm{x} = \lambda_1(H) \bm{x}$.

Since $G$ is connected, we see that $\bm{x}$ is positive. Recall that $A(G) \bm{x} = \lambda_1(G) \bm{x}$ and 
$A(H) \bm{x} = \lambda_1(H) \bm{x}$. We immediately obtain that
\begin{equation}\label{eq:diff-G-H}
(\lambda_1(G) - \lambda_1(H)) \bm{x} = (A(G) - A(H)) \bm{x}.
\end{equation}
Choose a vertex $w\notin \{u,v\}$ and consider the eigenvalue equations with respect to $w$. We deduce that
\[
(\lambda_1(G) - \lambda_1(H)) x_w = 0,
\]
which yields $\lambda_1(G)=\lambda_1(H)$. Combining with \eqref{eq:diff-G-H} we get a contradiction.

Finally, if $x_ux_v - \overline{x}_u\overline{x}_v > 0$ and $a_{uv} =0$, then $p(H)>p(G)$ by \eqref{eq:p(H)-vs-p(G)}, a 
contradiction. Conversely, if $a_{uv}=1$ and $x_ux_v - \overline{x}_u\overline{x}_v < 0$, then $p(H)>p(G)$ by 
\eqref{eq:p(H)-vs-p(G)}, a contradiction. This completes the proof of this lemma.
\end{proof}

\begin{lemma}\label{lem:bounds-xmax-zmax}
$x_{\max} = \Theta(n^{-1/2})$ and $\overline{x}_{\max} = \Theta(n^{-1/2})$.
\end{lemma}

\begin{proof}
Consider the complete bipartite graph $K_{\lfloor n/3\rfloor, \lceil 2n/3\rceil}$. One can check that 
$p(K_{\lfloor n/3\rfloor, \lceil 2n/3\rceil}) > 1.1n$. Thus, $\lambda_1 + \overline{\lambda}_1 > 1.1n$ 
due to the maximality. On the other hand, $\lambda_1 < n$ and $\overline{\lambda}_1 < n$,
we have $\lambda_1 = \Theta(n)$ and $\overline{\lambda}_1 = \Theta(n)$.

Let $u$ be a vertex such that $x_u = x_{\max}$. By eigenvalue equation and Cauchy--Schwarz inequality,
\[
\lambda_1^2 x_u^2 = \bigg(\sum_{uw\in E(G)} x_w\bigg)^2 \leq d(u) \sum_{uw\in E(G)} x_w^2 \leq d(u),
\]
which yields that 
\[
x_{\max} = x_u < \frac{\sqrt{d(u)}}{\lambda_1} = O\Big(\frac{1}{\sqrt{n}}\Big). 
\]
On the other hand, it is clear that $x_{\max}\geq n^{-1/2}$. Hence, we see $x_{\max} = \Theta(n^{-1/2})$. 
Likewise, we have $\overline{x}_{\max} = \Theta(n^{-1/2})$. 
\end{proof}

\begin{lemma}\label{lem:almost-equal}
For any pair of vertices $u$ and $v$, we have
\[
\big|(\lambda_1 x_u^2 + \overline{\lambda}_1 \overline{x}_u^2) - (\lambda_1 x_v^2 + \overline{\lambda}_1 \overline{x}_v^2)\big| 
= O\Big(\frac{1}{n}\Big).
\]
\end{lemma}

\begin{proof}
Let $H$ be the graph obtained from $G$ by deleting all edges incident with $u$, and adding edges $\{uw: w\in N_G(v)\}$. 
That is,
\[
E(H) = E(G) \setminus\{uw: w\in N_G(u)\} \cup \{uw: w\in N_G(v)\}.
\]
Define two vectors $\bm{z}$ and $\overline{\bm{z}}$ for $H$ by 
\[
z_w = 
\begin{cases}
x_w, & w\neq u, \\
x_v, & w = u, 
\end{cases}
\]
and 
\[
\overline{z}_w = 
\begin{cases}
\overline{x}_w, & w\neq u, \\
\overline{x}_v, & w = u.
\end{cases}
\]
Noting that $\lambda_1=\bm{x}^{\mathrm{T}} A(G) \bm{x}$ we deduce that
\begin{align*}
\bm{z}^{\mathrm{T}} A(H) \bm{z} - \lambda_1
& = 2 z_u \sum_{w\in N_G(v)\setminus \{u\}} x_w - 2 x_u \sum_{w\in N_G(u)} x_w \\
& = 2 x_v \sum_{w\in N_G(v)} x_w - 2 x_u \sum_{w\in N_G(u)} x_w - 2a_{uv} x_ux_v \\
& = 2\lambda_1 x_v^2 - 2\lambda_1 x_u^2 - 2a_{uv} x_ux_v.
\end{align*}
Similarly, we have
\[
\overline{\bm{z}}^{\mathrm{T}} A(\overline{H}) \overline{\bm{z}} - \overline{\lambda}_1
= 2\overline{\lambda}_1 \overline{x}_v^2 - 2\overline{\lambda}_1 \overline{x}_u^2 + 2\overline{x}_v^2 - 2(1-a_{uv}) \overline{x}_u \overline{x}_v.
\]
Combining the above equations gives
\begin{align*}
0 
& \geq \frac{\bm{z}^{\mathrm{T}} A(H) \bm{z}}{\bm{z}^{\mathrm{T}} \bm{z}}
+ \frac{\overline{\bm{z}}^{\mathrm{T}} A(\overline{H}) \overline{\bm{z}} }{\overline{\bm{z}}^{\mathrm{T}} \overline{\bm{z}}}
- (\lambda_1 + \overline{\lambda}_1) \\
& = \frac{\lambda_1 + 2\lambda_1 x_v^2 - 2\lambda_1 x_u^2 - 2a_{uv} x_ux_v}{1 - x_u^2 + x_v^2}
+ \frac{\overline{\lambda}_1 + 2 \overline{\lambda}_1 \overline{x}_v^2 - 2\overline{\lambda}_1 \overline{x}_u^2 + 2\overline{x}_v^2 - 2(1-a_{uv}) \overline{x}_u\overline{x}_v}{1 - \overline{x}_u^2 + \overline{x}_v^2}
- (\lambda_1 + \overline{\lambda}_1) \\
& = \frac{\lambda_1 x_v^2 - \lambda_1 x_u^2 - 2a_{uv} x_ux_v}{1 - x_u^2 + x_v^2}
+ \frac{\overline{\lambda}_1 \overline{x}_v^2 - \overline{\lambda}_1 \overline{x}_u^2 + 2 \overline{x}_v^2 - 2(1-a_{uv}) \overline{x}_u \overline{x}_v}{1 - \overline{x}_u^2 + \overline{x}_v^2}.
\end{align*}
By Lemma \ref{lem:bounds-xmax-zmax}, $x_u$, $x_v$, $\overline{x}_u$ and $\overline{x}_v$ are all $O(n^{-1/2})$. Therefore,
\[
(\lambda_1 x_v^2 - \lambda_1 x_u^2) + (\overline{\lambda}_1 \overline{x}_v^2 - \overline{\lambda}_1 \overline{x}_u^2) < O\Big(\frac{1}{n}\Big).
\]
Similarly, we also have 
\[
(\lambda_1 x_u^2 - \lambda_1 x_v^2) + (\overline{\lambda}_1 \overline{x}_u^2 - \overline{\lambda}_1 \overline{x}_v^2) < O\Big(\frac{1}{n}\Big).
\]
The desired result follows from the above two inequalities.
\end{proof}

The following theorem describes the approximate structure of the extremal graph $G$ except for $o(n)$ vertices, which are the 
main results in \cite{Terpai2011}. 

\begin{theorem}[\cite{Terpai2011}]\label{thm:approximate-structure}
Let $W$ and $\overline{W}$ be the graphons of $G$ and $\overline{G}$, respectively. Then for all $(x, y) \in [0, 1]^2$,
\[
W(x,y) = 
\begin{cases}
0, & (x,y)\in [1/3, 1]^2, \\
1, & \text{otherwise},
\end{cases}
\]
and the maximum eigenvalues of $W$ and $\overline{W}$ are $\mu = \overline{\mu} = 2/3$. Furthermore, if $f,g$ are nonnegative 
unit eigenfunctions associated to $\mu$, $\overline{\mu}$ respectively, then for every $x \in [0, 1]$,
\[
f(x) = 
\begin{cases}
\sqrt{2}, & x\in [0, 1/3], \\
\sqrt{2}/2, & \text{otherwise},
\end{cases}
\hspace{3mm}
g(x) =
\begin{cases}
0, & x\in [0, 1/3], \\
\sqrt{6}/2, & \text{otherwise}.
\end{cases}
\]
\end{theorem}

Combining Theorem \ref{thm:approximate-structure} and the arguments in \cite{BreenRiasanovksy2022} we obtain the following result.

\begin{lemma}\label{lem:approximate-structure}
The extremal graph $G$ is a complete split graph.
\end{lemma}

\begin{proof}
Let $\mu$ and $\overline{\mu}$ be the maximum eigenvalues of $W$ and $\overline{W}$, respectively. 
For every positive integer $n$, let $G_n$ denote a graph on $n$ vertices attaining 
$\max\{p(H): H\ \text{is an}\ n\text{-vertex graph}\}$, and let $W_n$ be any associated graphon 
corresponding to $G_n$. Denote by $\mu_n$, $\overline{\mu}_n$ the maximum eigenvalues
of $W_n$ and $\overline{W}_n$, respectively. By Proposition \ref{prop:converge}, 
\[
\lim_{n\to\infty} \mu(W_n) = \mu, \hspace{2mm}
\lim_{n\to\infty} \mu(\overline{W}_n) = \overline{\mu}.
\]
Furthermore, let $f_n$ be a nonnegative unit $\mu_n$-eigenfunction for $W_n$ and let $\overline{f}_n$ be a 
nonnegative unit $\overline{\mu}_n$-eigenfunction for $\overline{W}_n$. Moreover, we may apply measure-preserving 
transformations to each $W_n$ such that without loss of generality, $\|W-W_n\|_{\square}\to 0$. By Lemma 2 
of \cite{RuizChamon2020}, $f_n$ and $\overline{f}_n$ converge to $f$ and $\overline{f}$ in $\mathcal{L}^2$ sense, respectively. 

For convenience, we let $\alpha_1=\sqrt{2}$, $\beta_1=\sqrt{2}/2$, $\alpha_2=0$ and $\beta_2=\sqrt{6}/2$.
For any $\varepsilon_0 > 0$, we define 
\begin{align*}
\mathcal{S}_n & := \{x\in [0,1]: |f_n(x) - \alpha_1| < \varepsilon_0,~\, |\overline{f}_n(x) - \alpha_2| < \varepsilon_0\}, \\
\mathcal{L}_n & := \{x\in [0,1]: |f_n(x) - \beta_1| < \varepsilon_0,~\, |\overline{f}_n(x) - \beta_2| < \varepsilon_0\}, \\
\mathcal{T}_n & := [0,1]\setminus (\mathcal{S}_n \cup \mathcal{L}_n).
\end{align*}
One can show that the Lebesgue measure $m(\mathcal{T}_n)$ of $\mathcal{T}_n$ goes to zero, and $m(\mathcal{L}_n)\to 2/3$. 
Indeed, we have
\[
\int_{|f_n-f|\geq\varepsilon_0} |f_n - f|^2 \dif x \leq \int_0^1 |f_n - f|^2 \dif x \to 0,
\]
and
\[
\int_{|\overline{f}_n - \overline{f}|\geq\varepsilon_0} |\overline{f}_n - \overline{f}|^2 \dif x \leq \int_0^1 |\overline{f}_n - \overline{f}|^2 \dif x \to 0,
\]
as desired.

For all $u \in V (G_n)$, let $I_u$ be the subinterval of $[0, 1]$ corresponding to $u$ in $W_n$, and denote by 
$f_n(u)$ and $\overline{f}_n(u)$ the constant values of $f_n$ and $\overline{f}_n$ on $I_u$, respectively. For 
convenience, we define the following discrete analogues of $\mathcal{S}_n$, $\mathcal{L}_n$, $\mathcal{T}_n$:
\begin{align*}
S_n & := \{u\in V(G_n): |f_n(u) - \alpha_1| < \varepsilon_0,~\, |\overline{f}_n(u) - \alpha_2| < \varepsilon_0\}, \\
L_n & := \{u\in V(G_n): |f_n(u) - \beta_1| < \varepsilon_0,~\, |\overline{f}_n(u) - \beta_2| < \varepsilon_0\}, \\
T_n & := V(G_n)\setminus (S_n \cup L_n).
\end{align*}

For any $\varepsilon_1 > 0$. By Proposition \ref{prop:converge}, $\mu_n\to\mu$ and $\overline{\mu}_n\to\overline{\mu}$. 
It follows from Lemma \ref{lem:almost-equal} and Proposition \ref{prop:relation} that
\begin{equation}\label{eq:almost-identity}
|\mu f^2_n(u) + \overline{\mu}\, \overline{f}^2_n(u) - (\mu+\overline{\mu})| < \varepsilon_1,
\end{equation}
for sufficiently large $n$.

\begin{claim}\label{claim}
Let $\varepsilon_0' > 0$ and $n$ be sufficiently large in terms of $\varepsilon_0'$. For each $v\in T_n$, 
\[
\max\big\{\big|f_n(v) - \beta_1\big|,~~ \big|\overline{f}_n(v) - \beta_2 \big|\big\} < \varepsilon_0'.
\]
\end{claim}

\begin{proof}
Without loss of generality, we assume $G_n$ is connected. Hence, $\alpha_1 f_n(v) - \alpha_2 \overline{f}_n(v) > 0$. 
By Lemma \ref{lem:adjacent-iff}, $S_n\subseteq N_{G_n}(v)$ for sufficiently large $n$. Write 
\[ 
\mathcal{U}_n:=\bigcup_{w\in N_{G_n}(v)} I_w
\] 
and recall that $f_n$ is an eigenfunction of $W_n$ corresponding to $\mu_n$. Then
\begin{align*}
\mu_n f_n(v) 
& = \int_0^1 W_n(v, y) f_n(y) \dif y \\
& = \int_{y\in \mathcal{L}_n \cap\, \mathcal{U}_n} f_n(y) \dif y 
+ \int_{y\in \mathcal{S}_n} f_n(y) \dif y 
+ \int_{y\in \mathcal{T}_n \cap\, \mathcal{U}_n} f_n(y) \dif y.
\end{align*}
Denote the Lebesgue measure $m(\mathcal{L}_n\cap\mathcal{U}_n)$ of $\mathcal{L}_n\cap\mathcal{U}_n$ by $\gamma_n$. 
For any $\varepsilon_2 > 0$, if $n$ is sufficiently large and $\varepsilon_1$ is sufficiently small, then
\begin{equation}\label{eq:fn(v)}
\Big|f_n(v) - \frac{3\gamma_n\beta_1 + \alpha_1}{3\mu} \Big| < \varepsilon_2.
\end{equation}
Similarly, we have
\begin{align*}
\overline{\mu}_n \overline{f}_n(v) 
& = \int_0^1 \overline{W}_n(v,y) \overline{f}_n(y) \dif y \\
& = \int_{y\in \mathcal{L}_n \cap\, \overline{\mathcal{U}}_n} \overline{f}_n(y) \dif y 
+ \int_{y\in \mathcal{S}_n \cap\, \overline{\mathcal{U}}_n} \overline{f}_n(y) \dif y 
+ \int_{y\in \mathcal{T}_n \cap\, \overline{\mathcal{U}}_n} \overline{f}_n(y) \dif y,
\end{align*}
where $\overline{\mathcal{U}}_n := [0,1]\setminus\mathcal{U}_n$. Then
\begin{equation}\label{eq:overline-fn(v)}
\Big|\overline{f}_n(v) - \frac{(2-3\gamma_n)\beta_2}{3\overline{\mu}}\Big| < \varepsilon_2.
\end{equation}
Let $\varepsilon_3>0$. Combining \eqref{eq:fn(v)}, \eqref{eq:overline-fn(v)} and \eqref{eq:almost-identity}, we obtain 
\[
\Big|\mu\cdot\Big(\frac{3\gamma_n \beta_1 + \alpha_1}{3\mu}\Big)^2 + \overline{\mu}\cdot\Big(\frac{(2-3\gamma_n) \beta_2}{3\overline{\mu}}\Big)^2 - (\mu+\overline{\mu})\Big| < \varepsilon_3.
\]
Substituting the values of $\alpha_1$, $\alpha_2$, $\beta_1$, $\beta_2$ and simplifying, we deduce that
\[
|\gamma_n (3\gamma_n - 2)| < \varepsilon_3.
\]
It follows that if $n$ is sufficiently large and $\varepsilon_3$ is sufficiently small, then
\[
\min\{\gamma_n, |3\gamma_n-2|\} < \varepsilon_4
\]
for any $\varepsilon_4>0$. Combining with \eqref{eq:fn(v)} and \eqref{eq:overline-fn(v)}, we have either 
\[
\max\Big\{\Big|f_n(v) - \frac{\alpha_1}{3\mu}\Big|,~~ \Big|\overline{f}_n(v) - \frac{2\beta_2}{3\overline{\mu}}\Big|\Big\} < \varepsilon_0'
\]
or 
\[
\max\Big\{\Big|f_n(v) - \frac{\alpha_1+2\beta_1}{3\mu}\Big|,~~ |\overline{f}_n(v)|\Big\} < \varepsilon_0'.
\]
Notice that $(\alpha_1+2\beta_1)/(3\mu) = \alpha_1$ and $v\in T_n$, the second inequality does not hold.
So we obtain the desired claim by observing that $\alpha_1/(3\mu) = \beta_1$ and $2\beta_2/(3\overline{\mu}) = \beta_2$.
\end{proof}

Next, according to the definition of $S_n$ and $L_n$ and Claim \ref{claim}, we see for $n$ sufficiently large,
\begin{align*}
\max\{|f_n(v) - \alpha_1|,~~ |\overline{f}_n(v) - \alpha_2|\} < \varepsilon_0, & \hspace{3mm} v\in S_n \\
\max\{|f_n(v) - \beta_1|,~~ |\overline{f}_n(v) - \beta_2|\} < \varepsilon_0, & \hspace{3mm} v\in L_n \\
\max\{|f_n(v) - \beta_1|,~~ |\overline{f}_n(v) - \beta_2|\} < \varepsilon_0', & \hspace{3mm} v\in T_n.
\end{align*}
Let $\varepsilon_0$, $\varepsilon_0'$ be sufficiently small. Then for sufficiently large $n$ and for all $u,v\in V(G_n)$,
$f_n(u)f_n(v) > \overline{f}_n(u)\overline{f}_n(v)$ if and only if $u,v\in S_n$ or $(u,v)\in (S_n, (L_n\cup T_n)) \cup ((L_n\cup T_n), S_n)$.
Finally, using Lemma \ref{lem:adjacent-iff}, we get the required result.
\end{proof}

Now we are ready to prove Conjecture \ref{conj:NG-conjecture}. \par\vspace{2mm}

\noindent {\it Proof of Conjecture \ref{conj:NG-conjecture}.}
By Lemma \ref{lem:approximate-structure}, we assume $G = CS_{n,\omega}$. By easy algebraic computation, we find 
\[
\lambda_1(G) = \frac{\omega - 1 + \sqrt{-3\omega^2 + (4n-2)\omega + 1}}{2}.
\]
Therefore, we immediately have
\[
\lambda_1(G) + \lambda_1(\overline{G}) = \frac{\sqrt{-3\omega^2 + (4n-2)\omega + 1} - \omega}{2} + n-\frac{3}{2}.
\]
For convenience, denote $f(x) := \sqrt{-3x^2 + (4n-2)x + 1} - x$. To complete the proof, we need to determine the 
value $\omega$ maximizing $f(\omega)$. 

The following can be found in \cite{AouchicheBell2008}, we include it here for completeness. Letting $f'(x)=0$, 
we obtain 
\[
x_0 = \frac{2n-3 - \sqrt{n^2-n+1}}{3},~ 
x_1 = \frac{2n-3 + \sqrt{n^2-n+1}}{3}.
\]
It is easy to see that $f(\omega)$ attains its maximum value at $\lfloor x_0\rfloor$ or $\lceil x_0\rceil$.
Notice that 
\[
\lfloor x_0\rfloor = \frac{1}{3}
\begin{cases}
n - 3, & n \equiv 0 \hspace{-2.5mm} \pmod{3}, \\
n - 1, & n \equiv 1 \hspace{-2.5mm} \pmod{3}, \\
n - 2, & n \equiv 2 \hspace{-2.5mm} \pmod{3},   
\end{cases}
\hspace{2mm} \text{and} \hspace{2mm}
\lceil x_0\rceil = \frac{1}{3}
\begin{cases}
n, & n \equiv 0 \hspace{-2.5mm} \pmod{3}, \\
n + 2, & n \equiv 1 \hspace{-2.5mm} \pmod{3}, \\
n + 1, & n \equiv 2 \hspace{-2.5mm} \pmod{3}.   
\end{cases}
\]
Finally, we obtain the desired result by comparing $f(\lfloor x_0\rfloor)$ and $f(\lceil x_0\rceil)$. 
\hfill $\square$

\section{Connected graphs of maximum $Q$-spread}
\label{sec4}

For unit vectors $\bm{x},\bm{z}\in\mathbb{R}^n$, two basic inequalities can be obtained from the following well-known inequality for the Rayleigh quotient,
\[
q_1(G) \geq \bm{x}^{\mathrm{T}} Q(G) \bm{x} \hspace{3mm} \text{and} \hspace{3mm} 
q_n(G) \leq \bm{z}^{\mathrm{T}} Q(G) \bm{z}.
\] 
Hence, the $Q$-spread of a graph can be expressed as 
\begin{equation}\label{eq:max-min-Q-spread}
s_Q(G) = \max_{\bm{x}, \bm{z}} \sum_{uv\in E(G)} \big((x_u+x_v)^2 - (z_u+z_v)^2\big),
\end{equation}
where the maximum is taken over all unit vectors $\bm{x}$ and $\bm{z}$.

Throughout this section, we always assume that $G$ is a graph attaining maximum $Q$-spread among all connected $n$-vertex graphs.
Let $\bm{x}$ and $\bm{z}$ be the unit eigenvectors of $Q(G)$ corresponding to $q_1(G)$ and $q_n(G)$, respectively.
By the Perron--Frobenius theorem for nonnegative matrices, we may assume that $\bm{x}$ is a positive vector. We also set 
$q_1:=q_1(G)$ and $q_n:=q_n(G)$ for short.

\begin{lemma}\label{lem:q1-lower-bound}
$q_1 > 2n-5$ and $q_n < 3$.
\end{lemma}

\begin{proof}
Consider the graph $K_{n-1}^+$. Clearly, $q_1(K_{n-1}^+) > q_1(K_{n-1}) = 2(n-2)$. On the other hand, we have 
$q_n(K_{n-1}^+) \leq 1$ (see \cite{Das2010}). Indeed, we define a unit vector $\bm{y}$ for $K_{n-1}^+$ by
\[
y_u =
\begin{cases}
1, & d(u) = 1, \\
0, & \text{otherwise}.
\end{cases}
\]  
It follows that $q_n(K_{n-1}^+) \leq \bm{y}^{\mathrm{T}} Q(K_{n-1}^+) \bm{y} = 1$.
By maximality of $G$, we conclude that $s_Q(G) \geq s_Q(K_{n-1}^+) > 2n-5$.

Notice that $Q(G)$ is a positive semidefinite matrix. Hence, $q_1 \geq s_Q(G) > 2n-5$. On the other hand, 
it follows from $q_1 \leq 2\Delta(G)$ that $2n - 5 < q_1 - q_n \leq 2(n-1) - q_n$, which yields that 
$q_n < 3$.  This completes the proof of the lemma.
\end{proof}

\begin{lemma}\label{lem:size-lower-bound}
$|E(G)| > (n-1)(n-3)/2$.
\end{lemma}

\begin{proof}
By Theorem 3.1 in \cite{Das2004} and the main result in \cite{Merris1998}, we deduce that
\[
q_1 \leq \frac{2 |E(G)|}{n-1} + n-2.
\]
As shown in Lemma \ref{lem:q1-lower-bound}, $q_1 > 2n-5$. Therefore, 
\[
\frac{2 |E(G)|}{n-1} + n-2 > 2n-5,
\]
as desired.
\end{proof}

\begin{lemma}\label{lem:Q-adjacent-iff}
For any vertices $u$ and $v$, if $x_u + x_v > |z_u + z_v|$, then $u$ and $v$ are adjacent;
if $x_u + x_v < |z_u + z_v|$ and $G-uv$ is connected, then $u$ and $v$ are non-adjacent.
\end{lemma}

\begin{proof}
By \eqref{eq:max-min-Q-spread}, if $x_u + x_v > |z_u + z_v|$ and $u$, $v$ are non-adjacent, then
\[
s_Q(G+uv) - s_Q(G) \geq (x_u + x_v)^2 - (z_u + z_v)^2 > 0,
\]
a contradiction implies that $u$ and $v$ are adjacent. Likewise, if $x_u + x_v < |z_u + z_v|$ and $u$, $v$ 
are adjacent, then 
\[
s_Q(G - uv) - s_Q(G) \geq (z_u + z_v)^2 - (x_u + x_v)^2 > 0,
\]
a contradiction implies that $u$ and $v$ are non-adjacent. 
\end{proof}

\begin{lemma}\label{lem:x-max}
For each $v\in V(G)$, $x_v < \frac{\sqrt{n}}{n-3}$.
\end{lemma}

\begin{proof}
By eigenvalue equation and Cauchy--Schwarz inequality we have
\begin{align*}
q_1 x_v 
& = d(v) x_v + \sum_{u\in N(v)} x_u \\
& \leq (d(v) - 1) x_v + \sqrt{n} \\
& \leq (n-2) x_v + \sqrt{n},
\end{align*}
which, together with Lemma \ref{lem:q1-lower-bound}, gives the desired result.
\end{proof}

Fix a sufficiently small constant $\varepsilon > 0$, whose value will be chosen later. Let
\[
S:=\Big\{v\in V(G): |z_v| < \frac{\varepsilon}{\sqrt{n}}\Big\}, \hspace{3mm}
T:= \Big\{v\in V(G): x_v < \frac{1}{2\sqrt{n}}\Big\},
\]
and $L:= V(G)\setminus S$.

\begin{lemma}\label{lem:T-size}
$|T| <8$.
\end{lemma}

\begin{proof}
Since $\bm{x}$ is a unit vector, it follows from Lemma \ref{lem:x-max} that 
\[
1 = \sum_{u\in T} x_u^2 + \sum_{u\in V(G)\setminus T} x_u^2 < \frac{|T|}{4n} + (n - |T|) \cdot \frac{n}{(n-3)^2},
\]
yielding $|T| < 8$, as desired.
\end{proof}

\begin{lemma}\label{lem:unique-vetex-o(n)}
There is exactly one vertex with degree $o(n)$.
\end{lemma}

\begin{proof}
By Lemma \ref{lem:size-lower-bound}, we know that there exists at most one vertex with degree $o(n)$.
Hence, it suffices to rule out the case that $d(v)=\Omega(n)$ for each $v\in V(G)$. 

Let $u$ be a vertex such that $|z_u| = \max\{|z_w|: w\in V(G)\}$. Without loss of generality, we assume 
$z_u < 0$. Consider the following eigenvalue equation with respect to vertex $u$,
\[
(q_n - d(u)) z_u = \sum_{w\in N(u)} z_w.
\]
Noting that $q_n < 3$, $d(u)=\Omega(n)$ and $|z_u| \geq n^{-1/2}$, we have 
\[
\sum_{w\in N(u)} z_w = (q_n - d(u)) z_u > \Omega(\sqrt{n}).
\]
On the other hand, by Cauchy--Schwarz inequality we find that
\begin{align*}
\sum_{w\in N(u)} z_w 
& \leq \sum_{w\in N(u)\, \cap L} |z_w| + \sum_{w\in N(u) \setminus L} |z_w| \\
& \leq \sqrt{|L|} + \frac{\varepsilon}{\sqrt{n}} (n - |L|) \\
& \leq \sqrt{|L|} + \varepsilon\sqrt{n},
\end{align*}
which implies that $|L| > \Omega(n)$ for sufficiently small $\varepsilon$. Thus, one of the following two sets
\[
B:= \Big\{v\in L: z_v > 0\Big\},~~
C:= \Big\{v\in L: z_v < 0\Big\}
\]
(say $B$) has size $\Omega(n)$. Recall that $\bm{z}$ is a unit eigenvector corresponding to $q_n$. Hence,
\begin{align*}
q_n 
& = \sum_{uv\in E(G)} (z_u + z_v)^2 \\
& = \sum_{uv\in E(G)} (z_u^2 + z_v^2) + 2\sum_{uv\in E(G)} z_u z_v \\
& \geq 2 \sum_{uv\in E(G)} (|z_u z_v| + z_u z_v) \\
& = 2 \sum_{uv\in E(G),\, z_u z_v>0} (|z_u z_v| + z_u z_v) \\
& \geq 2 \sum_{uv\in E(B)} (|z_u z_v| + z_u z_v).
\end{align*}
By Lemma \ref{lem:size-lower-bound}, $|E(B)| > \Omega(n^2)$. It follows that $q_n > \Omega(n)$,
a contradiction completing the proof of Lemma \ref{lem:unique-vetex-o(n)}.
\end{proof}

From now on, we assume $w$ is the unique vertex such that $d(w) = o(n)$. 

\begin{lemma}\label{lem:x-Theta-sqrt-n}
For each $v\in V(G)\setminus \{w\}$, $x_v = \Theta(n^{-1/2})$.
\end{lemma}

\begin{proof}
Let $v$ be any vertex in $V(G)\setminus \{w\}$. Using the eigenvalue equation with respect to vertex $v$, we deduce that
\[
(q_1 - d(v)) x_v 
\geq \sum_{u\in N(v)\setminus T} x_u \geq \frac{d(v) - 8}{2\sqrt{n}} 
> \Omega(\sqrt{n}).
\]
The second inequality follows from Lemma \ref{lem:T-size}, and the last inequality is due to 
Lemma \ref{lem:unique-vetex-o(n)}. On the other hand, $(q_1 - d(v)) x_v < 2n x_v$. Therefore, 
$x_v > \Omega(n^{-1/2})$. Combining with Lemma \ref{lem:x-max} we get the desired result.
\end{proof}

\begin{lemma}\label{lem:z-o-sqrt-n}
For each $v\in V(G)\setminus \{w\}$, $|z_v| = o(n^{-1/2})$.
\end{lemma}

\begin{proof}
Suppose on the contrary that there exists a vertex $v\in V(G)\setminus \{w\}$ such that $|z_v|> c/\sqrt{n}$ for some $c>0$.
By eigenvalue equation and Lemma \ref{lem:unique-vetex-o(n)}, 
\begin{align*}
\bigg|\sum_{u\in N(v)} z_u \bigg| = |(q_n - d(v)) z_v| > \Omega(\sqrt{n}).
\end{align*}
On the other hand, we see
\[
\bigg|\sum_{u\in N(v)} z_u \bigg|
\leq \sum_{u\in N(v)\, \cap L} |z_w| + \sum_{u\in N(v) \setminus L} |z_u| 
\leq \sqrt{|L|} + \varepsilon\sqrt{n},
\]
which yields that $|L| > \Omega(n)$. Using similar arguments in the proof of Lemma \ref{lem:unique-vetex-o(n)}, we have 
$q_n > \Omega(n)$, a contradiction completing the proof.
\end{proof}

Now we are in a position to confirm the validity of the first part of Conjecture \ref{conj:Q-spectral-spread-conjecture}.
\par\vspace{2mm}

\noindent {\it Proof of Conjecture \ref{conj:Q-spectral-spread-conjecture}.}
By Lemma \ref{lem:x-Theta-sqrt-n} and Lemma \ref{lem:z-o-sqrt-n}, for any $u,v\in V(G)\setminus \{w\}$ we have 
$x_u + x_v > |z_u + z_v|$. Combining with Lemma \ref{lem:Q-adjacent-iff}, we deduce that the induce subgraph 
$G[V(G)\setminus\{w\}]$ is a complete graph. To complete the proof, we show that $d(w)=1$. Indeed,
by Lemma \ref{lem:z-o-sqrt-n} again, 
\[ 
z_w^2 = 1 - \sum_{u\in V(G)\setminus \{w\}} z_u^2 = 1 - o(1).
\]
If $d(w) \geq 2$, then deleting one edge between $w$ and $V(G)\setminus \{w\}$ will increase the $Q$-spread by 
Lemma \ref{lem:Q-adjacent-iff}, a contradiction completing the proof. \hfill $\square$

\end{document}